\documentclass[preprint]{elsarticle}

\usepackage{amssymb,amsmath}
\usepackage{amsthm}

\usepackage{lineno,hyperref}

\journal{ArXiv.org}

\newtheorem{thm}{Theorem}

\newtheorem{prop}[thm]{Proposition}
\newtheorem{corollary}[thm]{Corollary}
\newtheorem{conjecture}[thm]{Conjecture}
\newtheorem{definition}[thm]{Definition}
\newdefinition{rmk}{Remark}
\newproof{pf}{Proof}

\begin{document}

\begin{frontmatter}



\title{Schauder-Orlicz decompositions, $\ell_{\Phi}$-decompositions and pseudo-Daugavet property}

\author[l1,l2]{Vitalii Marchenko\corref{cor}}
\ead{vitalii.marchenko@usz.edu.pl, v.marchenko@ilt.kharkov.ua}
\cortext[cor]{Corresponding author}

 \affiliation[l1]{organization={Institute of Mathematics, University of Szczecin},
             addressline={Wielkopolska 15},
             city={Szczecin},
             postcode={70451},
             state={Poland}}

 \affiliation[l2]{organization={B. Verkin Institute for Low Temperature Physics and Engineering of the National Academy of Sciences of Ukraine},
             addressline={Nauky Avenue 47},
             city={Kharkiv},
             postcode={61103},
             state={Ukraine}}



\begin{abstract}
The concept of $\ell_{\Phi}$-decomposition, extending the concept of  $\ell_{p}$-decomposition of a Banach space, is presented and basic properties of Schauder-Orlicz decompositions and $\ell_{\Phi}$-decompositions are studied.
We show that Schauder-Orlicz decompositions are orthogonal in a sense of Grinblyum-James and Singer.
Simple constructions of $\ell_{p}$-decompositions and Schauder-Orlicz decompositions in $L_p$ are presented.
We prove that in the class of spaces possessing pseudo-Daugavet property, which includes classical $L_p$, $1\leq p\neq 2$, and $C$,
Schauder-Orlicz decompositions with at least one finite dimensional subspace do not exist.
It follows that Kato theorem on similarity for sequences of projections~\cite{Kato} cannot be extended to spaces from this class.
Moreover we show that Banach spaces, possessing Schauder-Orlicz decompositions with at least one finite dimensional subspace, do not have pseudo-Daugavet property.
Thus for Banach spaces $X$ possessing Schauder-Orlicz decompositions we obtain the following characterization of pseudo-Daugavet property: $X$ has pseudo-Daugavet property if and only if there is no Schauder-Orlicz decomposition in $X$ with at least one finite dimensional subspace if and only if there is no Schauder-Orlicz decomposition in $X$, which is an FDD.
\end{abstract}




\end{frontmatter}


\section{Introduction}

Let $X$ be a Banach space with norm $\|\cdot\|.$
The sequence $\{\mathfrak{M}_n\}_{n=0}^{\infty}$ of closed nonzero linear subspaces of $X$ is called a Schauder decomposition (or basis of subspaces) of $X$ provided that
each $x\in X$ has a unique, norm convergent
expansion $x=\sum\limits_{n=0}^{\infty} x_n,$ where $x_n\in \mathfrak{M}_n$ for all $n,$ i.e.
\begin{equation}\label{SD}
X=\bigoplus\limits_{n=0}^{\infty} \mathfrak{M}_n.
\end{equation}
It has long been known that every normed linear space has a decomposition~\cite{Sanders} (the case when all subspaces from~(\ref{SD}) need not to be closed), but $\ell_{\infty}$ has no Schauder decomposition~\cite{Dean},~\cite{Singer2}.
Nevertheless, if we denote $E=\ell_{\infty}=c_0^{\ast \ast}$, then each of spaces $E$, $E^{\ast \ast}$, $E^{\ast\ast \ast \ast}$, $\ldots$ has no Schauder decomposition whereas
all of spaces $c_0$, $c_0^{\ast}=\ell_1$, $E^{\ast }$, $E^{\ast\ast \ast}$, $E^{\ast\ast \ast\ast \ast}$, $\ldots$ have Schauder decompositions, see Example 15.1 in~\cite{Singer2}.

If $\{\mathfrak{M}_n\}_{n=0}^{\infty}$ is a Schauder decomposition of $X$, the sequence
of corresponding bounded linear projections $\{P_n\}_{n=0}^{\infty}$ on $X$, defined by $P_n x=x_n,$  where $x=\sum\limits_{n=0}^{\infty} x_n,$ $x_n\in \mathfrak{M}_n$, will be called, as in~\cite{Singer2} and~\cite{Marchenko1},
the associated sequence of coordinate projections (a.s.c.p.). Thus, if $\{\mathfrak{M}_n\}_{n=0}^{\infty}$ is a Schauder decomposition of $X$, then each $x\in X$ has a unique expansion $$x=\sum\limits_{n=0}^{\infty} P_n x,$$ where $\{P_n\}_{n=0}^{\infty}$ is the a.s.c.p. satisfying $P_i P_j=\delta_{i}^{j} P_i$ for all indexes $i,j\geq 0.$
It is well known that only separable Banach space may possess a Schauder basis, whereas there exist nonseparable Banach spaces with Schauder decompositions, see e.g.~\cite{Sanders},~\cite{Chadwick}.
It is worth noting that every non-separable reflexive Banach space has a Schauder decomposition~\cite{Chadwick}.
On the other hand, there exists a separable Banach space without a Schauder decomposition~\cite{Allexandrov}.

In case when every $\mathfrak{M}_n$ is finite dimensional, the decomposition is called a finite
dimensional decomposition, or simply an FDD. Every FDD is clearly a Schauder decomposition and it was shown by Johnson that a separable Banach space $X$ has an FDD if and only if $X$ is a dual $\pi_{\lambda}$ space~\cite{Johnson}. Note that there exist Banach spaces with FDD, which have no Schauder basis~\cite{Szarek}.

Schauder decompositions are frequently used in functional analysis and operator theory.
Among the most recent papers we point out one by Hajek and Medina~\cite{Hajek} dealing mainly with separable Banach spaces, where
it was proved that if a Banach space has an FDD then there exists a Lipschitz retraction onto a small generating convex and compact set.
Moreover, it was proved in~\cite{Hajek} that if a Banach space possesses a small generating Lipschitz retract, then it has the
$\pi$-property.
It is worth mentioning that a Banach space with the metric $\pi$-property has an FDD~\cite{Johnson}, but it is still an open question whether the $\pi$-property implies the  metric $\pi$-property. Thus these results combining together constitute almost a characterization of the existence of an FDD for separable Banach spaces in terms of the existence of Lipschitz retracts.
For the case of dual spaces this characterization is indeed true~\cite{Hajek}.
Linearization of holomorphic functions and composition ideals of bounded holomorphic mappings were studied via absolute Schauder decompositions in~\cite{Botelho}.
In~\cite{Kruse} Schauder decompositions were used to the study of series representations in spaces of vector-valued functions.
In~\cite{DelaRosa} the authors proved that any separable complex infinite dimensional Banach space with an unconditional Schauder decomposition has chaotic and frequently hypercyclic operator.
Furthermore, the existence of chaotic and frequently hypercyclic polynomials of arbitrary positive degree in spaces with mentioned in previous sentence properties was shown in~\cite{Bernandes}.
Note that there exists a real Banach space with unconditional basis and without chaotic operators~\cite{DelaRosa}.
For more about Schauder decompositions and bases we refer to classical monographs~\cite{Johnson2},~\cite{Johnson3},~\cite{L-T},~\cite{Singer1},~\cite{Singer2}.

In~\cite{Marchenko1} the author generalized a theorem of Kato~\cite{Kato} to the case of isomorphic Schauder decompositions in Banach spaces possessing certain geometric structure.
To this end $\ell_{\Psi}$-Hilbertian and $\infty$-Hilbertian Schauder decompositions were used instead of initial orthogonal Schauder decompositions in a Hilbert space from Kato theorem.
Moreover, the concept of an orthogonal Schauder decomposition was generalized to the case of Banach spaces (new concept was called in~\cite{Marchenko1} by Schauder-Orlicz decomposition) and a class of Banach spaces with Schauder-Orlicz decompositions was considered.
\begin{definition}\label{SOD}\cite{Marchenko1}
A Schauder decomposition $\{\mathfrak{M}_n\}_{n=0}^{\infty}$ of $X$ with the a.s.c.p. $\{P_n\}_{n=0}^{\infty}$ is called a Schauder-Orlicz decomposition with an Orlicz function $\Phi$  provided there exists an Orlicz function $\Phi$ such that for every $x\in X$
\begin{equation}\label{Schauder-Orlicz}
    \inf\left\{\rho >0: \sum\limits_{n=0}^{\infty} \Phi\left(\frac{\|P_n x\|}{\rho}\right)\leq 1\right\} = \|x\|.
\end{equation}
\end{definition}
Note that all separable Orlicz spaces $\ell_{\Phi}$, i.e. such that $\Phi$ satisfies the $\Delta_2$-condition at zero~\cite{L-T}, have a Schauder-Orlicz decomposition with an Orlicz function $\Phi$. Indeed, it is sufficient to consider a Schauder decomposition $\{\mathfrak{M}_n\}_{n=0}^{\infty}$ corresponding to a canonical basis $\{e_n\}_{n=0}^{\infty}$, i.e. $\{\mathfrak{M}_n=Lin \{e_n\}\}_{n=0}^{\infty}$.
Thus, $\ell_p$-spaces have a Schauder-Orlicz decomposition
 with an Orlicz function $t^p$ for any $p\geq 1.$
Moreover, by the generalized Parseval identity, we have that a Schauder decomposition in a Hilbert space $H$ is orthogonal if and only if it is a Schauder-Orlicz decomposition with an Orlicz function $\Phi(t)=t^2$.
Therefore the concept of Schauder-Orlicz decomposition is natural generalization of the concept of an orthogonal Schauder decomposition to the case of Banach spaces.

In Section $2$ we present the concept of $\ell_{\Phi}$-decomposition, extending the concept of  $\ell_{p}$-decomposition of a Banach space, see Definitions~\ref{ellphi} and~\ref{ellphic}.
The question on the existence of Schauder-Orlicz decomposition or $\ell_{\Phi}$-decomposition in a Banach space seems to be new and perspective, since it is linked with an existence of certain geometric structure in a Banach space and,
as it turns out by Theorems~\ref{seqofspace} and~\ref{seqofspaceSOD},
is related to the study of infinite $\ell_{\Phi}$-sums of Banach spaces.
On the other hand, an existence of such structure in a Banach space $X$ may be useful for various applications, similar to the role played by orthogonal and Riesz bases of subspaces in a Hilbert space,
 e.g. in operator theory, spectral theory or in the study of evolution equations on $X$~\cite{Zwart},~\cite{SM}.

The main purpose of this work is the study of concepts of Schauder-Orlicz decomposition and $\ell_{\Phi}$-decomposition as well as relations between them.
The next section is focused  on properties of Schauder-Orlicz decompositions and $\ell_{\Phi}$-decompositions, see Theorem~\ref{bc}, Corollaries~\ref{bcsod},~\ref{rel},  Theorems~\ref{seqofspace},~\ref{isometric},~\ref{seqofspaceSOD},~\ref{uncond}, Propositions~\ref{ellphi1} and~\ref{uniq}.
In particular, Theorem~\ref{seqofspaceSOD} yields an answer on our question from~\cite{Marchenko1} about the characterization of Banach spaces possessing a Schauder-Orlicz decomposition.
In Theorem~\ref{orthogonal} we prove that each Schauder-Orlicz decomposition is monotone and orthogonal in a sense of Grinblyum-James and Singer, thereby once more confirm that the concept of
Schauder-Orlicz decomposition is natural extension of
orthogonality concept to the case of Banach spaces.
However, it turns out that the concept of Grinblyum-James orthogonal Schauder decomposition is much more general than the concept of Schauder-Orlicz decomposition, see Proposition~\ref{subset}.

In Section $3$ we give a simple construction of Schauder-Orlicz decomposition with an Orlicz function $t^p$ in space $L_p$, $1\leq p<\infty$, -- a slicing Schauder decomposition, see Theorem~\ref{SODLp}.
Thus we give a positive answer to our question from~\cite{Marchenko1} on the existence of Schauder-Orlicz decomposition in $L_p$.
However, each subspace of slicing Schauder decomposition of $L_p$ is infinite dimensional and isometric to $L_p$.
Also the construction of $\ell_{p}$-decomposition in $L_p$, based on the construction of Schauder-Orlicz decomposition from Theorem~\ref{SODLp}, is given by Corollary~\ref{conlp}.
Furthermore, the absence of other types of $\ell_{\Phi}$-decompositions and Schauder-Orlicz decompositions in $L_p$ is proved in Theorem~\ref{existLp}.
Similar result on the absence of other types of $\ell_{\Phi}$-decompositions, except $c_0$-decompositions, in space $C$ is given by Theorem~\ref{existC}.

It is clear that spaces $\ell_p$, $p\geq 1$, and Hilbert spaces have a Schauder-Orlicz decomposition, which is at the same time an FDD.
This fact was the cornerstone of extensions of Kato stability theorem~\cite{Kato} to stability theorems for unconditional Schauder decompositions in $\ell_p$ spaces~\cite{Marchenko3} and to stability theorems for Riesz basis of subspaces in a Hilbert space,
at least one of which is finite dimensional~\cite{Marchenko2}.
Note that to obtain stability theorems for unconditional Schauder decompositions in $\ell_p$ an intrinsic geometric properties of $\ell_p$ and the best constants in the Khintchine inequality were essentially used~\cite{Marchenko3}.
Actually to obtain a version of Kato theorem for a Banach space $X$ one should know that $X$ has a Schauder-Orlicz decomposition
 with at least one finite dimensional subspace, see Theorem~7.2 from~\cite{Marchenko1}.
However, it is proved in Section $4$ that spaces
$L_p$, $p\neq 2$, and $C$, in contrast to spaces $\ell_p$, $p\geq 1$, do not have  a Schauder-Orlicz decomposition with at least one finite dimensional subspace, see Theorems~\ref{nonexistence} and~\ref{nonexistenceC}.
Thus, it turns out that analogs of Kato theorem cannot be obtained in these spaces.
Theorem~\ref{normal} shows that normality of a Schauder decomposition in a Hilbert space means that it is a Schauder-Orlicz decomposition with an Orlicz function $t^2$.
However, as a corollary of Theorem~\ref{nonexistence} we obtain that there exists a normal Schauder decomposition in $L_p$, $p\neq 2$, which is not a Schauder-Orlicz decomposition
(see Corollary~\ref{corex}).
Also in Theorem~\ref{best} we deduce from the well-known optimality of the Haar basis~\cite{Olevskii},~\cite{Johnson2}
the best estimate for constants of one dimensional unconditional Schauder decompositions in  $L_p$, $1<p<\infty.$
The main result of Section $4$ is Theorem~\ref{nonexistenceX}, first part of which demonstrates that the validity of pseudo-Daugavet property in a Banach space $X$ immediately implies the
lack of Schauder-Orlicz decomposition in $X$ with at least one finite dimensional subspace.
It means that Kato theorem on similarity for sequences of projections~\cite{Kato} cannot be extended to spaces with pseudo-Daugavet property.
Conversely, the second part of Theorem~\ref{nonexistenceX} reads that  existence in $X$ of the Schauder-Orlicz decomposition, at least one subspace of which is finite dimensional, ensures that
 $X$ does not have pseudo-Daugavet property.
For Banach spaces $X$ possessing Schauder-Orlicz decompositions we obtain the criterion of $X$ to have pseudo-Daugavet property on the language of dimensions of elements of Schauder-Orlicz decompositions, which exist in $X$,  see Theorem~\ref{crit}.

Finally, Section $5$ contains concluding remarks and indicates some further perspectives, including the Conjecture~\ref{conj} on the
uniqueness, up to an isometry, of Schauder-Orlicz decomposition with preassigned dimensions in $\ell_1$ and $c_0$.

\section{Schauder-Orlicz decompositions and $\ell_{\Phi}$-decompositions}

\subsection{Basic properties of Schauder-Orlicz decompositions and $\ell_{\Phi}$-decompositions}
It is known that every $c_0$-decomposition is shrinking, every $\ell_1$-decomposition is boundedly complete and every $\ell_p$-decomposition,
where $1<p<\infty$, is both shrinking and boundedly complete~\cite{Singer2}. Also in~\cite{Sandersd} Sanders proved that a Banach space $X$ with Schauder decomposition
$\{\mathfrak{M}_n\}_{n=0}^{\infty}$ is reflexive if and only if $\{\mathfrak{M}_n\}_{n=0}^{\infty}$ is both shrinking and boundedly complete and each $\mathfrak{M}_n$
is reflexive.
Finite dimensional $\ell_p$-decompositions, where $1\leq p<\infty$, as well as $c_0$-decompositions were studied by Johnson $\&$ Zippin~\cite{Johnson4}.
Three years later Chadwick~\cite{ChadwickA} introduced the concept of absolute Schauder decomposition (coincides with $\ell_1$-decomposition), generalizing the concept of absolute Schauder basis.  Also in~\cite{ChadwickA} it was proved that space possessing $\ell_1$-decomposition has a complemented subspace with a Schauder basis.

Throughout the paper we will use the following definition, which substantially extends the concept of $\ell_p$-decomposition and includes the concept of $c_0$-decomposition.
\begin{definition}\label{ellphi}
Let $\Phi$ be an Orlicz function.
  A Schauder decomposition, which is both  $\ell_{\Phi}$-Hilbertian and $\ell_{\Phi}$-Besselian, will be called an $\ell_{\Phi}$-decomposition.
\end{definition}
Obviously every Schauder-Orlicz decomposition is an $\ell_{\Phi}$-decomposition, but the converse is false.
Moreover, every Riesz basis of subspaces in a Hilbert space is an $\ell_2$-decomposition  and vice versa~\cite{Marchenko1}.

For $\ell_{\Phi}$-decompositions we can prove the following result.
\begin{thm}\label{bc}
Let $\{\mathfrak{M}_n\}_{n=0}^{\infty}$ be a Schauder decomposition of $X$.

(i) If $\{\mathfrak{M}_n\}_{n=0}^{\infty}$ is $\ell_{\Phi}$-decomposition with non-degenerate  Orlicz function $\Phi$, then
 $\{\mathfrak{M}_n\}_{n=0}^{\infty}$ is boundedly complete.

(ii) If $\{\mathfrak{M}_n\}_{n=0}^{\infty}$ is $\ell_{\Phi}$-decomposition with degenerate  Orlicz function $\Phi$, then
 $\{\mathfrak{M}_n\}_{n=0}^{\infty}$ is shrinking.
\end{thm}
\begin{proof}
 Let $\{\mathfrak{M}_n\}_{n=0}^{\infty}$ be $\ell_{\Phi}$-decomposition with non-degenerate  Orlicz function $\Phi$.
 Consider the a.s.c.p. $\{P_n\}_{n=0}^{\infty}$ of $\{\mathfrak{M}_n\}_{n=0}^{\infty}$.
 Then~\cite{Marchenko1},~\cite{Singer2} there exist constants $c_2\geq c_1>0$ such that for each $x=\sum\limits_{n=0}^{\infty} P_n x\in X$ we have
 \begin{equation}\label{Hilbert-Bessel}
c_1 \|\{\|P_n x\|\}_{n=0}^{\infty}\|_{\ell_{\Phi}} \leq \|x\|\leq c_2  \|\{\|P_n x\|\}_{n=0}^{\infty}\|_{\ell_{\Phi}}.
\end{equation}

Now take any sequence $\{x_i\}_{i=0}^{\infty}\in X$, $x_i\in \mathfrak{M}_i$ for each $i$, such that
$$\left\{\left\| \sum\limits_{i=0}^{n} x_i\right\| \right\}_{n=0}^{\infty}$$ is bounded, i.e. there exists a constant $c>0$ such that for all $n$ we have
$$\left\| \sum\limits_{i=0}^{n} x_i\right\| \leq c.$$
Applying of left-hand the side of~(\ref{Hilbert-Bessel}) yields
$$c_1 \|\:\{\|x_i\|\}_{i=0}^{n}\:\|_{\ell_{\Phi}}=c_1 \|\:\{\|P_i x\|\}_{i=0}^{n}\:\|_{\ell_{\Phi}} \leq \left\| \sum\limits_{i=0}^{n} x_i\right\| \leq c,$$
hence the sequence of norms $\{\|\:\{\|x_i\|\}_{i=0}^{n}\:\|_{\ell_{\Phi}}\}_{n=0}^{\infty}$ is bounded and increasing.
Consequently the sequence $\{\|\:\{\|x_i\|\}_{i=0}^{n}\:\|_{\ell_{\Phi}}\}_{n=0}^{\infty}$ is convergent. Combining this with the right-hand the side of~(\ref{Hilbert-Bessel}) yields
that for any $m>n$
$$\|x_{m+n}-x_n\| = \left\| \sum\limits_{i=n+1}^{m+n} x_i\right\|\leq c_2  \|\:\{\|P_i x\|\}_{i=n+1}^{m+n}\:\|_{\ell_{\Phi}}=c_2  \|\:\{\|x_i\|\}_{i=n+1}^{m+n}\:\|_{\ell_{\Phi}}\to 0$$
as $n,m\to \infty.$
Thus $\{x_i\}_{i=0}^{\infty}$ is a Cauchy sequence in $X$, hence $\{x_i\}_{i=0}^{\infty}$ is convergent and (i) is proved.

To prove (ii) we note that if the Orlicz function $\Phi$ is degenerate, then $\ell_{\Phi}$ is isomorphic to $\ell_{\infty}$~\cite{L-T}.
It follows that if $\{\mathfrak{M}_n\}_{n=0}^{\infty}$ is an $\ell_{\Phi}$-decomposition with degenerate  Orlicz function $\Phi$, then
$\{\mathfrak{M}_n\}_{n=0}^{\infty}$ is $c_0$-decomposition.
Hence  $\{\mathfrak{M}_n\}_{n=0}^{\infty}$ is shrinking~\cite{Singer2}.
\end{proof}

For Schauder-Orlicz decompositions we have the following corollary of Theorem~\ref{ellphi}.
\begin{corollary}\label{bcsod}
 Let $\{\mathfrak{M}_n\}_{n=0}^{\infty}$ be a Schauder-Orlicz decomposition with non-degenerate  Orlicz function $\Phi$.
 Then $\{\mathfrak{M}_n\}_{n=0}^{\infty}$ is boundedly complete.
\end{corollary}

By result of Sanders~\cite{Sandersd} we have that  every Schauder-Orlicz decomposition in reflexive space $X$ is shrinking.
The construction of non-shrinking Schauder-Orlicz decomposition of $L_1(0,1)$ is given by Theorem~\ref{SODLp}.

Further we have the following natural generalization of Corollary~15.6 of~\cite{Singer2}.
\begin{thm}\label{seqofspace}
A Banach space $X$ has an $\ell_{\Phi}$-decomposition
  if and only if there exists a sequence of Banach spaces $\{X_n\}_{n=0}^{\infty}$ such that
  $$X\cong \left(X_0,X_1,X_2,\ldots \right)_{\ell_{\Phi}},$$
  where $\left(X_0,X_1,X_2,\ldots \right)_{\ell_{\Phi}}$ stands for a Banach space of all sequences $\mathcal{X}=\{x_n\}_{n=0}^{\infty}$, such that $x_n\in X_n$ and
$$|||\mathcal{X}||| =\inf\left\{\rho >0: \sum\limits_{n=0}^{\infty} \Phi\left(\frac{\|x_n\|}{\rho}\right)\leq 1\right\}<\infty.$$
\end{thm}
\begin{proof}
  To prove this result one can use, e.g., arguments from the proof of Corollary~15.6 of~\cite{Singer2} together with two-sided estimate~(\ref{Hilbert-Bessel}).
\end{proof}
Note that for the case of $\ell_{1}$-decompositions, i.e. when $\Phi(t)=t$, Theorem~\ref{seqofspace} was first proved by Chadwick~\cite{ChadwickA}.

It turns out that $\ell_{\Phi}$-decomposition is invariant under isomorphism and Schauder-Orlicz decomposition is invariant under isometry.
\begin{thm}\label{isometric}
Let $\{\mathfrak{M}_n\}_{n=0}^{\infty}$ be a Schauder decomposition in a Banach space $X$.
\begin{enumerate}
  \item Suppose that $\{\mathfrak{M}_n\}_{n=0}^{\infty}$ is a Schauder-Orlicz decomposition with an Orlicz function $\Phi$ and the a.s.c.p. $\{P_n\}_{n=0}^{\infty}$ in $X$
  and $X$ is isometric to a Banach space $Y$ with the operator of isometry $T:X\mapsto Y$.
  Then $\{T \mathfrak{M}_n\}_{n=0}^{\infty}$ is also a Schauder-Orlicz decomposition with an Orlicz function $\Phi$ in $Y$.
  \item Suppose that $\{\mathfrak{M}_n\}_{n=0}^{\infty}$ is an $\ell_{\Phi}$-decomposition with an Orlicz function $\Phi$ and the a.s.c.p. $\{P_n\}_{n=0}^{\infty}$ in $X$
  and $X$ is isomorphic to a Banach space $Y$ with the isomorphism $S:X\mapsto Y$.
  Then $\{S \mathfrak{M}_n\}_{n=0}^{\infty}$ is also an $\ell_{\Phi}$-decomposition with an Orlicz function $\Phi$ in $Y$.
\end{enumerate}
\end{thm}
\begin{proof}
1. For all $x\in X$ we have that $Tx=y\in Y$, $x=T^{-1}y$ and, moreover,
  $$\|Tx\|_{Y}=\|x\|_{X},\qquad \|P_n x\|_{X} = \|T P_n x\|_{Y}, \: n\geq 0.$$
  Since $\{\mathfrak{M}_n\}_{n=0}^{\infty}$ is a Schauder decomposition of $X$ with the a.s.c.p. $\{P_n\}_{n=0}^{\infty}$ we have that every $x\in X$ has a unique expansion
  $$\sum\limits_{n=0}^{\infty} P_n x =x.$$
  Hence
  $$\sum\limits_{n=0}^{\infty} T P_n x =Tx.$$
  Thus every $y=Tx\in Y$ has a unique expansion
  $$\sum\limits_{n=0}^{\infty} T P_n T^{-1} y =TT^{-1} y = y,$$
  which means that $\{T P_nT^{-1} Y\}_{n=0}^{\infty}$  is a Schauder decomposition of $Y$.

  To show that $\{T P_nT^{-1} Y\}_{n=0}^{\infty}$ is a Schauder-Orlicz decomposition with an Orlicz function $\Phi$ in a Banach space $Y$ we note that
  for any $y\in Y$ there exists $x=T^{-1}y\in X$ such that
  \begin{align*}
    \|y\|_{Y}=\|Tx\|_{Y}=  \|x\|_{X}  =  &\inf\left\{\rho >0: \sum\limits_{n=0}^{\infty} \Phi\left(\frac{\|P_n x\|_{X}}{\rho}\right)\leq 1\right\}\\
    =  &\inf\left\{\rho >0: \sum\limits_{n=0}^{\infty} \Phi\left(\frac{\|T P_n x\|_{Y}}{\rho}\right)\leq 1\right\}\\
    = &\inf\left\{\rho >0: \sum\limits_{n=0}^{\infty} \Phi\left(\frac{\|T P_n T^{-1}y\|_{Y}}{\rho}\right)\leq 1\right\}.
  \end{align*}
  Hence~\cite{Marchenko1} $\{T \mathfrak{M}_n\}_{n=0}^{\infty}$ is also a Schauder-Orlicz decomposition with an Orlicz function $\Phi$ and the a.s.c.p. $\{T P_nT^{-1}\}_{n=0}^{\infty}$  in $Y$.

  The proof of the second statement goes similar to the proof of the first.
\end{proof}

\begin{corollary}\label{rel}
  Let $\{\mathfrak{M}_n\}_{n=0}^{\infty}$ be a Schauder-Orlicz decomposition with an Orlicz function $\Phi$
  in a Banach space $X$ and $X$ is isomorphic to a Banach space $Y$ with the isomorphism $S:X\mapsto Y$.
  Then $\{S \mathfrak{M}_n\}_{n=0}^{\infty}$ is an $\ell_{\Phi}$-decomposition in $Y$.
\end{corollary}
\begin{proof}
  The arguments of the proof are the same as in the proof of Theorem~\ref{isometric}.
\end{proof}

One can prove a partial case of Theorem~\ref{seqofspace}.
\begin{thm}\label{seqofspaceSOD}
A Banach space $X$ has a Schauder-Orlicz decomposition with an Orlicz function $\Phi$
  if and only if there exists a sequence of Banach spaces $\{X_n\}_{n=0}^{\infty}$ such that space $X$ is isometric to
  $$\left(X_0,X_1,X_2,\ldots \right)_{\ell_{\Phi}}.$$
\end{thm}

In what follows we are using the following definitions.
\begin{definition}\label{MUSD}\cite{Marchenko1}
  A Schauder decomposition $\{\mathfrak{M}_n\}_{n=0}^{\infty}$ in a Banach space $X$ is called
unconditional with constant $M$ provided that
 $\exists M\geq 1$ such that
$$ \left\|\sum\limits_{i=0}^{n} \delta_i y_i\right\| \leq M \left\|\sum\limits_{i=0}^{n}  y_i\right\|$$
for any $n\geq 0$, $y_i\in \mathfrak{M}_i$ and selectors $\delta_i\in \{0,1\}.$
\end{definition}

\begin{definition}\label{ellphic}
  We will call an $\ell_{\Phi}$-decomposition $\{\mathfrak{M}_n\}_{n=0}^{\infty}$ of a Banach space $X$ with  the a.s.c.p. $\{P_n\}_{n=0}^{\infty}$
  by $\ell_{\Phi}$-decomposition with constants $c_1,$ $c_2$ provided there exist constants $c_2\geq c_1>0$ such that for each $x=\sum\limits_{n=0}^{\infty} P_n x\in X$ we have~(\ref{Hilbert-Bessel}), i.e.
$$c_1 \|\{\|P_n x\|\}_{n=0}^{\infty}\|_{\ell_{\Phi}} \leq \|x\|\leq c_2  \|\{\|P_n x\|\}_{n=0}^{\infty}\|_{\ell_{\Phi}}.$$
\end{definition}
Clearly any Schauder-Orlicz decomposition with an Orlicz function $\Phi$ is an $\ell_{\Phi}$-decomposition with constants $1$, $1$.

\begin{thm}\label{uncond}
  Every $\ell_{\Phi}$-decomposition $\{\mathfrak{M}_n\}_{n=0}^{\infty}$ with constants $c_1,$ $c_2$ in $X$ is unconditional with constant  $\frac{c_2}{c_1}$.
  Every Schauder-Orlicz decomposition is unconditional with constant $1.$
\end{thm}
\begin{proof}
Fix any set of selectors $\delta_i\in \{0,1\}$.
  Since $\{\mathfrak{M}_n\}_{n=0}^{\infty}$ with the a.s.c.p. $\{P_n\}_{n=0}^{\infty}$ is an $\ell_{\Phi}$-decomposition in $X$,  for each $x=\sum\limits_{n=0}^{\infty} x_n=\sum\limits_{n=0}^{\infty} P_n x\in X$ we have
    applying~(\ref{Hilbert-Bessel}) that
    \begin{align*}
       \left\|\sum\limits_{i=0}^{\infty} \delta_i x_i\right\| & = \left\|\sum\limits_{i=0}^{\infty} \delta_i P_i x\right\|\leq c_2  \left\|\left\{\left\|P_j \sum\limits_{i=0}^{\infty} \delta_i P_i x\right\|\right\}_{j=0}^{\infty}\right\|_{\ell_{\Phi}}=c_2  \left\|\left\{\left\|\delta_j P_j  x\right\|\right\}_{j=0}^{\infty}\right\|_{\ell_{\Phi}}\\
     &= c_2  \left\|\left\{\delta_j\left\| P_j  x\right\|\right\}_{j=0}^{\infty}\right\|_{\ell_{\Phi}}\leq  c_2  \left\|\left\{\left\| P_j  x\right\|\right\}_{j=0}^{\infty}\right\|_{\ell_{\Phi}}
     \leq \frac{c_2}{c_1} \|x\|.
    \end{align*}
    It follows that $\{\mathfrak{M}_n\}_{n=0}^{\infty}$ is unconditional with constant $\frac{c_2}{c_1}$.

    To prove that every Schauder-Orlicz decomposition in a Banach space is unconditional with constant $1$ it is enough to note that Schauder-Orlicz decomposition is an $\ell_{\Phi}$-decomposition such that $c_1=c_2=1$ in~(\ref{Hilbert-Bessel}).
\end{proof}

Note that spaces $C[0,1]$ and $L_1(0,1)$ have no unconditional basis, but they still have unconditional Schauder decompositions~\cite{Singer2}.
An example of unconditional Schauder decomposition (even Schauder-Orlicz decomposition) in space $L_1(0,1)$ may be found in Theorem~\ref{SODLp}.
Moreover, $L_1(0,1)$ has no absolutely convergent basis~\cite{Singer1} although
it has an absolute Schauder decomposition (coincides with notion of $\ell_1$-decomposition), see Theorem~\ref{SODLp}.
The reader can find more about
various  properties of unconditional Schauder decompositions in Banach spaces in, e.g.,~\cite{Singer2},~\cite{L-T},~\cite{Marchenko1},~\cite{Marchenko2},~\cite{Marchenko3}.

\subsection{Inclinations of subspaces, Grinblyum-James orthogonality and Schauder-Orlicz decompositions}
Let $\mathfrak{M}$ and $\mathfrak{N}$ be two subspaces of a Banach space $X$.
Recall that an inclination of $\mathfrak{M}$ to $\mathfrak{N}$ is the number~\cite{Gurarii}
$$\left(\widehat{\mathfrak{M},\mathfrak{N}}\right)=\inf\limits_{x\in \mathfrak{M}, y\in \mathfrak{N}, \|x\|=1} \|x-y\|.$$

We mention here some facts about inclinations of subspaces and their importance in Schauder decompositions theory.
It is easy to see that
$$\left(\widehat{\mathfrak{M},\mathfrak{N}}\right)=\frac{1}{\|P\|},$$
where $P$ is a projection from $\mathfrak{M} \oplus \mathfrak{N}$ onto $\mathfrak{M}$, parallel to $\mathfrak{N}$~\cite{GurariiO}.
Since $\|P\|\geq 1$
it follows that for any $\mathfrak{M}$, $\mathfrak{N}$ we have
$$0\leq \left(\widehat{\mathfrak{M},\mathfrak{N}}\right)\leq 1.$$
Note that the notion of inclination is not symmetric in general~\cite{GurariiOR}: there exists $\mathfrak{M}$, $\mathfrak{N}$ such that
$\left(\widehat{\mathfrak{M},\mathfrak{N}}\right)\neq \left(\widehat{\mathfrak{N},\mathfrak{M}}\right)$. Moreover, the following is true~\cite{GurariiOR}:
\begin{enumerate}
  \item If $\left(\widehat{\mathfrak{M},\mathfrak{N}}\right)=\delta$, then
  $\left(\widehat{\mathfrak{N},\mathfrak{M}}\right)\geq \frac{\delta}{1+\delta}$. The latter inequality is exact for any $\delta\in[0,1].$
  \item $\left(\widehat{\mathfrak{M},\mathfrak{N}}\right)=\left(\widehat{\mathfrak{N},\mathfrak{M}}\right)$ for any two subspaces $\mathfrak{M}$, $\mathfrak{N}$
  in a Banach space $X$ with $\dim X>2$ if and only if $X$ is a space with a scalar product.
\end{enumerate}

The direct sum  $\mathfrak{M} \oplus \mathfrak{N}$ of two subspaces in a Banach space is closed if and only if~\cite{GrinblyumG},~\cite{GurariiOR}
$$\left(\widehat{\mathfrak{M},\mathfrak{N}}\right)>0.$$

The following geometric criterion of Schauder decomposition was proved by Grinblyum in~\cite{Grinblyum},~\cite{GrinblyumG}.
\begin{thm}\cite{Grinblyum,GrinblyumG,Gurarii}\label{gg}
Let $\{\mathfrak{M}_n\}_{n=0}^{\infty}$ be
 a complete sequence of subspaces in a Banach space $X$.
 Then $\{\mathfrak{M}_n\}_{n=0}^{\infty}$ is a Schauder decomposition of $X$ if and only if
 there exists a constant $\delta>0$ such that for all $n\geq 0$ and $m\geq 1$ we have
 \begin{equation}\label{infd}
   \left(\widehat{\mathfrak{G}_n,\mathfrak{L}_{n,m}}\right)\geq \delta.
 \end{equation}
 Here
 $$\mathfrak{G}_n=\overline{Lin}\left\{\mathfrak{M}_0,\mathfrak{M}_1,\ldots,\mathfrak{M}_n\right\},$$
 $$\mathfrak{L}_{n,m}=\overline{Lin}\left\{\mathfrak{M}_{n+1},\mathfrak{M}_{n+2},\ldots,\mathfrak{M}_{n+m}\right\}.$$
\end{thm}

The supremum of all $\delta$, satisfying~(\ref{infd}), was called in~\cite{Grinblyum} the index of the sequence $\{\mathfrak{M}_n\}_{n=0}^{\infty}$.
The index $\gamma\left(\{\mathfrak{M}_n\}_{n=0}^{\infty}\right)$ is uniquely defined the sequence $\{\mathfrak{M}_n\}_{n=0}^{\infty}$ and satisfy~\cite{Grinblyum}
$$0\leq \gamma\left(\{\mathfrak{M}_n\}_{n=0}^{\infty}\right)\leq1.$$
Moreover, if
\begin{equation}\label{orthogonalg}
  \gamma\left(\{\mathfrak{M}_n\}_{n=0}^{\infty}\right)=1,
\end{equation}
then the sequence $\{\mathfrak{M}_n\}_{n=0}^{\infty}$ is called orthogonal~\cite{Gurarii}.

If $\left(\widehat{\mathfrak{M},\mathfrak{N}}\right)=1$ then it is said that $\mathfrak{M}$ is orthogonal to $\mathfrak{N}$~\cite{Gurarii}.
In~\cite{Gurarii} Gurarii proved the following interesting geometric result: There exists Banach spaces of arbitrary finite dimension $n>2$, for each two subspaces $\mathfrak{M}$, $\mathfrak{N}$  of which, where
$\dim \mathfrak{M} >1$, we have
$$\left(\widehat{\mathfrak{M},\mathfrak{N}}\right)<1.$$
It follows that there exists Banach spaces of arbitrary finite dimension $n>2$ without an orthogonal Schauder basis~\cite{Gurarii}, i.e. a basis, linear spans of elements of which satisfy~(\ref{orthogonalg}), see also~\cite{Singer1}, p.~233. This notion of orthogonality of bases in Banach spaces is due to Grinblyum~\cite{Grinblyum} and James~\cite{James}.
\begin{rmk}
Note that for any Schauder basis in Banach space $X$ one can define a new norm, which depends on this basis and equivalent to the original norm of $X$,
with respect to which this basis become orthogonal~\cite{Banach}.
\end{rmk}

\begin{prop}\label{schauder basis}
  There exists an orthogonal Schauder decomposition in $C[a,b]$ which is not Schauder-Orlicz decomposition.
\end{prop}
\begin{proof}
  There exists an orthogonal basis in space $C[a,b]$~\cite{Singer3}, e.g. a classical basis of Schauder~\cite{Schauder}, see also~\cite{Singer1}.
  Since unconditional bases do not exist in $C[a,b]$, this basis is conditional.
  Denote by $\{\mathfrak{M}_n\}_{n=0}^{\infty}$ its one dimensional linear spans. Then obviously $\{\mathfrak{M}_n\}_{n=0}^{\infty}$ is, on the one hand, orthogonal.
  On the other hand, by virtue of Theorem~\ref{uncond} we have that
  $\{\mathfrak{M}_n\}_{n=0}^{\infty}$ is not a Schauder-Orlicz decomposition.
\end{proof}

For any Schauder-Orlicz decomposition in a Banach space $X$ we have the following result.

\begin{thm}\label{orthogonal}
Let $\{\mathfrak{M}_n\}_{n=0}^{\infty}$ be a Schauder-Orlicz decomposition with an Orlicz function $\Phi$  in $X$. Then the following is true.
\begin{enumerate}
\item $\{\mathfrak{M}_n\}_{n=0}^{\infty}$ is monotone.
  \item For any $n,m$ and all $x_{i_j}\in \mathfrak{M}_{i_j}$ $(j=0,\ldots, n)$,  $x_{l_j}\in \mathfrak{M}_{l_j}$ $(j=0,\ldots, m)$,
  satisfying $\{i_1,\ldots, i_n\} \cap \{l_1,\ldots, l_m\}=\emptyset$, we have
  $$ \left\| \sum\limits_{j=0}^{n} x_{i_j} \right\| \leq \left\| \sum\limits_{j=0}^{n} x_{i_j} + \sum\limits_{j=0}^{m} x_{l_j}\right\|$$
  (i.e.  $\{\mathfrak{M}_n\}_{n=0}^{\infty}$ is Singer-orthogonal, see remark 3 of the last section).
  \item For any $n,m$ such that $n\neq m$ we have $(\widehat{\mathfrak{M}_n,\mathfrak{M}_m})=1.$
  \item For any $n\geq 0$ and $m\geq 1$ we have
  $$\left(\widehat{\mathfrak{G}_n,\mathfrak{L}_{n,m}}\right)=1,$$
  where subspaces $\mathfrak{G}_n$ and $\mathfrak{L}_{n,m}$ are defined in Theorem~\ref{gg}.

  Thus $\gamma\left(\{\mathfrak{M}_n\}_{n=0}^{\infty}\right)=1$, i.e.  $\{\mathfrak{M}_n\}_{n=0}^{\infty}$ is orthogonal.
\end{enumerate}
\end{thm}
\begin{proof}
Consider any Schauder-Orlicz decomposition $\{\mathfrak{M}_n\}_{n=0}^{\infty}$  with an Orlicz function $\Phi$ and the a.s.c.p. $\{P_n\}_{n=0}^{\infty}$ in a Banach space $X$.
  Then, by definition of Schauder-Orlicz decomposition for all $n,m$ we have
  \begin{align*}
   \left\|\sum\limits_{j=0}^{n} y_j\right\| &= \inf\left\{\rho >0: \sum\limits_{j=0}^{n} \Phi\left(\frac{\|y_j\|}{\rho}\right)\leq 1\right\} \\
    & \leq \inf\left\{\rho >0: \sum\limits_{j=0}^{n+m} \Phi\left(\frac{\|y_j \|}{\rho}\right)\leq 1\right\}=\left\|\sum\limits_{j=0}^{n+m} y_j\right\|
  \end{align*}
  for any $y_j\in \mathfrak{M}_j,$ $j=1,\ldots,n+m$,
  which means that $\{\mathfrak{M}_n\}_{n=0}^{\infty}$ is monotone.

The proof of the second statement is straightforward.

Denote by  $\{P_n\}_{n=0}^{\infty}$ the a.s.c.p. of $\{\mathfrak{M}_n\}_{n=0}^{\infty}$.
To prove the third statement we note that for any $j,m$ such that $j\neq m$
\begin{align*}
  (\widehat{\mathfrak{M}_j,\mathfrak{M}_m}) & = \inf\limits_{x\in \mathfrak{M}_j, y\in \mathfrak{M}_m, \|x\|=1} \|x-y\|\\
   & =\inf\limits_{x\in \mathfrak{M}_j, y\in \mathfrak{M}_m, \|x\|=1} \|\{\|P_n(x-y)\|\}_{n=0}^{\infty}\|_{\ell_{\Phi}} \\
   & = \inf\limits_{x\in \mathfrak{M}_j, y\in \mathfrak{M}_m, \|x\|=1} \|\{\|P_n(P_j x-P_my)\|\}_{n=0}^{\infty}\|_{\ell_{\Phi}}\\
   & =  \inf\limits_{x\in \mathfrak{M}_j, y\in \mathfrak{M}_m, \|x\|=1} \|\left(0,\ldots,0,\smash[b]{\! \underbrace{\|x\|\,}_\text{$j^{th}$ place}},0,\ldots,0,\smash[b]{\! \underbrace{\|y\|\,}_\text{$m^{th}$ place}},0, \ldots\right)\|_{\ell_{\Phi}}\\
   &=1.
\end{align*}

Furthermore, since $\{\mathfrak{M}_n\}_{n=0}^{\infty}$ is a Schauder-Orlicz decomposition in $X$ with the a.s.c.p. $\{P_n\}_{n=0}^{\infty}$ we have for any $j\geq 0$ and $m\geq 1$  that
\begin{align*}
  & \left(\widehat{\mathfrak{G}_j,\mathfrak{L}_{j,m}}\right)= \inf\limits_{x\in \mathfrak{G}_j, y\in \mathfrak{L}_{j,m}, \|x\|=1} \|x-y\|\\
  &\qquad\:=\inf\limits_{x\in \mathfrak{G}_j, y\in \mathfrak{L}_{j,m}, \|x\|=1} \|\{\|P_n(x-y)\|\}_{n=0}^{\infty}\|_{\ell_{\Phi}} \\
   & =  \inf\limits_{x\in \mathfrak{M}_j, y\in \mathfrak{M}_m, \|x\|=1} \|\left(\|P_0 x\|,\ldots,\|P_jx\|,\|P_{j+1}y\|, \ldots ,\|P_{j+m}y\|,0,0, \ldots\right)\|_{\ell_{\Phi}}\\
   &\qquad\:=1,
\end{align*}
 where subspaces $\mathfrak{G}_n$ and $\mathfrak{L}_{n,m}$ are defined in Theorem~\ref{gg}.
 Hence $\{\mathfrak{M}_n\}_{n=0}^{\infty}$ is orthogonal.
\end{proof}

Combining Theorem~\ref{orthogonal} with Proposition~\ref{schauder basis} we arrive at the following observation.
\begin{prop}\label{subset}
  The set of all Schauder-Orlicz decompositions is strictly contained in the set of all Grinblyum-James orthogonal Shauder decompositions.
\end{prop}

\subsection{To the question of uniqueness of $\ell_{\Phi}$-decompositions and Schauder-Orlicz decompositions}

For any $\ell_{\Phi}$-decomposition in a Banach space $X$ we have, using ideas of the proof of Theorem~\ref{orthogonal}, the following result.

\begin{prop}\label{ellphi1}
Let $\{\mathfrak{M}_n\}_{n=0}^{\infty}$ be an $\ell_{\Phi}$-decomposition with constants $c_1,$ $c_2$ in $X$. Then the following is true.
\begin{enumerate}
  \item For any $n,m$ and all $x_{i_j}\in \mathfrak{M}_{i_j}$ $(j=0,\ldots, n)$,  $x_{l_j}\in \mathfrak{M}_{l_j}$ $(j=0,\ldots, m)$,
  satisfying $\{i_1,\ldots, i_n\} \cap \{l_1,\ldots, l_m\}=\emptyset$, we have
  $$ \left\| \sum\limits_{j=0}^{n} x_{i_j} \right\| \leq \frac{c_2}{c_1} \left\| \sum\limits_{j=0}^{n} x_{i_j} + \sum\limits_{j=0}^{m} x_{l_j}\right\|.$$
  \item For all $n,m$ and  for any $y_j\in \mathfrak{M}_j,$ $j=1,\ldots,n+m$, we have that
  $$\left\|\sum\limits_{j=0}^{n} y_j\right\|\leq \frac{c_2}{c_1} \left\|\sum\limits_{j=0}^{n+m} y_j\right\|.$$
\end{enumerate}

\end{prop}

Next we discuss the uniqueness of Schauder-Orlicz decompositions and $\ell_{\Phi}$-decompositions in Hilbert spaces.
\begin{prop}\label{uniq}
  (i) Every Schauder-Orlicz decomposition with preassigned dimensions in a Hilbert space is unique, up to an isometry.

  (ii) Every  $\ell_2$-decomposition with preassigned dimensions in
a Hilbert space is unique, up to an isomorphism.
\end{prop}
\begin{proof}
Every  Schauder decomposition in a Hilbert space is unconditional if and only if it is a Riesz basis of subspaces if and only if it is an $\ell_2$-decomposition~\cite{Marchenko1}.
Since unconditional Schauder decomposition with preassigned dimensions in a Hilbert space
 is unique, up to an isomorphism~\cite{Marchenko1},~\cite{Marchenko2}, we have (ii).

To prove (i) it remains to note that
  Schauder decomposition in a Hilbert space is orthogonal if and only if it is a Schauder-Orlicz decomposition with Orlicz function $t^2$~\cite{Marchenko1}.
\end{proof}

\begin{rmk}\label{exist}
Note that by virtue of Theorem~\ref{uncond} in a Hilbert space from the set of all $\ell_{\Phi}$-decompositions only $\ell_2$-decompositions exist.
Furthermore, by virtue of Theorem~\ref{uncond} in a Hilbert space only Schauder-Orlicz decompositions with an Orlicz function $t^2$ exist.
\end{rmk}

\begin{rmk}
In~\cite{LindenstraussPelcz} Lindenstrauss and Pelczynski found that all unconditional Schauder bases
in $\ell_1$ (resp. $c_0$) are equivalent to the canonical basis of $\ell_1$ (resp. $c_0$).
Moreover, one year later Lindenstrauss and Zippin established that the only Banach spaces, which have a unique, up to equivalence, normalized unconditional basis are $\ell_2$, $\ell_1$
and $c_0$~\cite{LindenstraussZipp}.

Also in~\cite{LindenstraussPelcz}  the following remarkable result was proved (see Corollary~$8$ in~\cite{LindenstraussPelcz}):
If $X$ is an $\mathcal{L}_1$-space (respectively $\mathcal{L}_{\infty}$-space) and if
$\{\mathfrak{M}_n\}_{n=0}^{\infty}$ is an unconditional Schauder decomposition of $X$, then
$$X\cong \left(\mathfrak{M}_0,\mathfrak{M}_1,\mathfrak{M}_2,\ldots \right)_{\ell_1},$$
or
$$X\cong \left(\mathfrak{M}_0,\mathfrak{M}_1,\mathfrak{M}_2,\ldots \right)_{c_0},$$
respectively.
Particularly it means that all unconditional Schauder decompositions in $\ell_1$ or $L_1(0,1)$ are $\ell_1$-decompositions and all unconditional Schauder decompositions in $c_0$ or
$C[0,1]$ are
$c_0$-decompositions.
\end{rmk}

\section{Existence of Schauder-Orlicz decompositions and $\ell_p$-decompositions in spaces $L_p$, $1\leq p<\infty$}

Since all separable $L_p$, $1\leq p<\infty,$  with purely non-atomic measures are isometric to $L_p(0,1)$ with standard Lebesgue measure~\cite{Beauzamy} and since by Proposition~\ref{isometric}
 Schauder-Orlicz decompositions are invariant under isometry,
we focus ourselves on spaces $L_p(0,1)$, $1\leq p<\infty,$ with standard Lebesgue measure $\mu$.
Construction of Schauder-Orlicz decompositions in other types of spaces $L_p$ may be done similarly to the following.
\begin{thm}\label{SODLp}
  Consider the family of sets $\left\{A_n \right\}_{n=0}^{\infty}$ with the following properties:
\begin{enumerate}
  \item  $A_i\cap A_j=\emptyset$ if  $i\neq j$;
  \item $\bigcup\limits_{n=0}^{\infty} A_n=(0,1)$;
  \item $\mu(A_n)>0$ for all $n$ and $\mu(A_n)\to 0$ as $n\to \infty$.
\end{enumerate}
Further fix any $p\geq 1$ and consider the corresponding projections $P_n$ in $L_p(0,1)$. For every $f\in L_p(0,1)$ we define
$$P_n f(t)= \left \{
\begin{array}{l}
f(t),\quad t \in A_n,\\
0,\quad\text{otherwise}.\\
\end{array}
\right .$$

Then $\left\{P_n L_p(0,1)\right\}_{n=0}^{\infty}$ is a Schauder-Orlicz decomposition with an Orlicz function $t^p$  in $L_p(0,1)$.
\end{thm}
\begin{proof}
  Indeed, it is easy to check that each linear subspace $P_n L_p(0,1)$, $n\geq 0,$ is closed.
  Further, the fact that for every $f\in L_p(0,1)$
  $$\sum\limits_{n=0}^{\infty} P_n f =f$$
  is obvious by construction.
  Finally, by the additivity of Lebesgue integral we have that for all  $f\in L_p(0,1)$
  $$\sum\limits_{n=0}^{\infty} \left\|P_n f\right\|^p = \sum\limits_{n=0}^{\infty} \int\limits_{A_n} |f(t)|^p dt = \int\limits_{0}^1 |f(t)|^p dt = \|f\|^p.$$
  This means that $\left\{P_n L_p(0,1)\right\}_{n=0}^{\infty}$ is a Schauder-Orlicz decomposition with an Orlicz function $t^p$  in $L_p(0,1)$.
\end{proof}

Taking into account their nature it is convenient to call Schauder decompositions from Theorem~\ref{SODLp} slicing Schauder decompositions.

\begin{rmk}
Each subspace $P_n L_p(0,1)$, $n\geq 0,$ is isometric to $L_p(0,1)$ and, hence, infinite dimensional.
So we give a simple construction of Schauder-Orlicz decomposition in $L_p$, each subspace of which is infinite dimensional.
\end{rmk}

Applying Theorem~\ref{seqofspaceSOD} we have the following simple corollary of Theorem~\ref{SODLp}.
\begin{corollary}
  The space $L_p$, $1\leq p<\infty,$ is isometric to
  $$\left(\mathfrak{M}_0,\mathfrak{M}_1,\mathfrak{M}_2,\ldots \right)_{\ell_{p}},$$
  where $\left\{\mathfrak{M}_n\right\}_{n=0}^{\infty}$ is slicing Schauder-Orlicz decomposition with an Orlicz function $t^p$  in $L_p$.
\end{corollary}

\begin{rmk}
While writing this paper the author found out that
the simplest slicing Schauder decomposition
appeared first in a paper~\cite{Chadwick} (Example 3.7) but in a different context.
It was constructed as a boundedly complete Schauder decomposition in $L_1(0,1)$.
\end{rmk}

Theorem~\ref{SODLp} yields also the constructions of $\ell_p$-decompositions in $L_p(0,1)$, $1\leq p<\infty$.
\begin{corollary}\label{conlp}
  Let $\{\mathfrak{M}_n\}_{n=0}^{\infty}$ be a Schauder-Orlicz decomposition from Theorem~\ref{SODLp} in $L_p(0,1)$.
  Assume that $S$ is an isomorphism in $L_p(0,1)$.
  Then $\{S \mathfrak{M}_n\}_{n=0}^{\infty}$ is an $\ell_p$-decomposition in $L_p(0,1)$.
\end{corollary}
\begin{proof}
  The proof follows from Corollary~\ref{rel}.
\end{proof}

We can obtain the following result on the existence of  $\ell_{\Phi}$-decompositions and Schauder-Orlicz decompositions in spaces $L_p$, $1\leq p<\infty$, similar to the situation in Hilbert spaces, see Remark~\ref{exist}.
\begin{thm}\label{existLp}
   From the set of all $\ell_{\Phi}$-decompositions in $L_p$  only $\ell_p$-decompositions exist.
Also in $L_p$ only Schauder-Orlicz decompositions with an Orlicz function $t^p$ exist.
\end{thm}
\begin{proof}
  The proof of this theorem is based on the combination of Theorem~\ref{uncond} with Proposition~4.8 from~\cite{Marchenko1} on two-sided estimates for unconditional Schauder decompositions
  in $L_p$ spaces, see also Corollary~2.4 in~\cite{Marchenko3}.
\end{proof}

Combining Theorem~\ref{uncond} with Corollary~8 of~\cite{LindenstraussPelcz} we arrive at the following conclusion.
\begin{thm}\label{existC}
   From the set of all $\ell_{\Phi}$-decompositions in $C$  only $c_0$-decompositions exist.
\end{thm}

However, the construction of Schauder-Orlicz decomposition in the space $C$ of continuous functions looks as a non-trivial problem, since the slicing decomposition in $C[0,1]$ is obviously not a Schauder decomposition in $C[0,1]$.
Thus, even the existence of a Schauder-Orlicz decomposition in $C[0,1]$ is unclear, although $C[0,1]$ has a $c_0$-decomposition~\cite{Singer2}.

\section{Nonexistence of Schauder-Orlicz decompositions with at least one finite dimensional subspace}

\subsection{Nonexistence of Schauder-Orlicz decompositions with at least one finite dimensional subspace in spaces $L_p$, $1\leq p\neq 2$ and $C$}
Since $L_p \ncong\ell_p$ for $p\neq 2$,  it is clear that there is no basis in $L_p$, linear spans of elements of which form a Schauder-Orlicz decomposition.
Further we prove much more strong statement that there is no Schauder-Orlicz decomposition in $L_p$, $p\neq 2$, with at least one finite dimensional subspace.
\begin{thm}\label{nonexistence}
  There is no Schauder-Orlicz decomposition in  $L_p$, $1\leq p\neq 2$, with at least one finite dimensional subspace.
\end{thm}
\begin{proof}
  By Theorem~\ref{existLp} only Schauder-Orlicz decomposition with an Orlicz function $t^p$ can exist in $L_p$.

  First we consider the case when $p\in(1,2)\cup(2,\infty)$
  and assume that there exists the Schauder-Orlicz decomposition  $\{\mathfrak{M}_n\}_{n=0}^{\infty}$ with an Orlicz function $t^p$ in $L_p$, such that
  \begin{equation}\label{dim}
    \dim \mathfrak{M}_0<\infty.
  \end{equation}
Denote by   $\{P_n\}_{n=0}^{\infty}$ the a.s.c.p. of $\{\mathfrak{M}_n\}_{n=0}^{\infty}$.
Then we note that~(\ref{dim}) means that $P_0$ is compact. Clearly $\|P_0\|\geq 1.$
Let
$$\varepsilon = \|P_0\|\geq 1.$$
Then, on the one hand, by Theorem~2 from~\cite{Benyamini} for chosen above $\varepsilon>0$ there exists $\delta(\varepsilon)>0$ such that
\begin{equation}\label{es1}
  \left\|I-P_0\right\|= \left\|\sum\limits_{n=0}^{\infty} P_n -P_0\right\|= \left\|\sum\limits_{n=1}^{\infty} P_n \right\|\geq 1+ \delta(\varepsilon)>1.
\end{equation}

On the other hand, since $\{\mathfrak{M}_n\}_{n=0}^{\infty}$ is a Schauder-Orlicz decomposition with an Orlicz function $t^p$, for any $x\in L_p$ we have
$$\sum\limits_{n=0}^{\infty} \left\| P_n x\right\|^p=\|x\|^p.$$
Hence for any $x\in L_p$ we have
$$\left\|\sum\limits_{n=1}^{\infty} P_nx \right\|^p = \sum\limits_{j=0}^{\infty} \left\|P_j \left(\sum\limits_{n=1}^{\infty} P_n x\right) \right\|^p
= \sum\limits_{n=1}^{\infty} \left\| P_n x\right\|^p \leq \sum\limits_{n=0}^{\infty} \left\| P_n x\right\|^p =\|x\|^p.$$
Thus we arrive at
\begin{equation}\label{es2}
  \left\|\sum\limits_{n=1}^{\infty} P_n \right\| \leq 1.
\end{equation}
Combining~(\ref{es1}) with (\ref{es2}) we arrive at a contradiction.

For the case $p=1$ the proof goes the same as above.
It is based on the Daugavet equation in $L_1$,  proved in~\cite{Babenko}. It reads that for any compact operator $K$ in $L_1$ one has
$$\|I+K\|=1+\|K\|.$$
\end{proof}

\begin{rmk}
Theorem~\ref{nonexistence} has an essential geometric nature, since for the case $1<p\neq 2$ it is based on Theorem~2 from~\cite{Benyamini}, which in turn is essentially based on Clarkson's inequalities for norms in $L_p$ (see Lemma~3 in~\cite{Benyamini}).
\end{rmk}

Also we have the following obvious corollary of Theorem~\ref{nonexistence} concerning FDDs in $L_p$.
\begin{corollary}\label{fddL}
  There is no Schauder-Orlicz decomposition in $L_p$, $1\leq p\neq 2$, which is an FDD.
\end{corollary}

Clearly each Schauder-Orlicz decomposition with the a.s.c.p. $\{P_n\}_{n=0}^{\infty}$ is normal, i.e. for all $n$ we have
$$\|P_n\|=1,$$
but the converse if false, in general. We can prove this fact as a corollary of Theorem~\ref{nonexistence}.

\begin{corollary}\label{corex}
  There exists a normal Schauder decomposition in $L_p(0,1)$,  $1\leq p\neq 2$ which is not a Schauder-Orlicz decomposition.
\end{corollary}
\begin{proof}
  To show this it enough to consider the normalized Haar basis  $\{h_n\}_{n=0}^{\infty}$ in $L_p(0,1)$,  $1\leq p\neq 2$, and corresponding to it one dimensional projections
$\{J_n\}_{n=0}^{\infty}$. It is known that $\{h_n\}_{n=0}^{\infty}$ is a normal basis of $L_p(0,1)$ for any $ p\geq 1$~\cite{Singer1}, so
$$\|J_n\|=1$$
for each $n$. Let $\mathfrak{M}_n=Lin\{h_n\}$ for any $n.$ Then clearly $\dim \mathfrak{M}_n=1$ for any $n.$
Suppose that $\{\mathfrak{M}_n\}_{n=0}^{\infty}$ is a Schauder-Orlicz decomposition in $L_p(0,1)$,  $1\leq p\neq 2$.
Then we arrive at a contradiction with Theorem~\ref{nonexistence}.
\end{proof}

However, in a Hilbert space class of normal Schauder decompositions coincides with the class of Schauder-Orlicz decompositions with an Orlicz function $t^2$.
\begin{thm}\label{normal}
  A Schauder decomposition in a Hilbert space is normal if and only if it is a Schauder-Orlicz decomposition with an Orlicz function $t^2$.
\end{thm}
\begin{proof}
  Clearly, a Schauder decomposition $\{\mathfrak{M}_n\}_{n=0}^{\infty}$ of a Hilbert space is a Schauder-Orlicz decomposition with an Orlicz function $t^2$ if and only if $\{\mathfrak{M}_n\}_{n=0}^{\infty}$ is orthogonal.
  Further, $\{\mathfrak{M}_n\}_{n=0}^{\infty}$ is orthogonal if and only if the a.s.c.p. $\{P_n\}_{n=0}^{\infty}$ of $\{\mathfrak{M}_n\}_{n=0}^{\infty}$ consists only of orthogonal projections.
  Since the criterion for a projection $P$ in a Hilbert space to be an orthogonal projection is
  $$\|P\|=1,$$
   $\{\mathfrak{M}_n\}_{n=0}^{\infty}$ is orthogonal if and only if it is normal.
\end{proof}

Note that if $C_p(\{h_n\}_{n=0}^{\infty})$ is an unconditional constant of the Haar basis in $L_p(0,1)$, $1<p<\infty$, and $C_p(\{u_n\}_{n=0}^{\infty})$ is an unconditional constant of any other unconditional basis $\{u_n\}_{n=0}^{\infty}$ of $L_p(0,1)$, $1<p<\infty$, then
$$C_p(\{h_n\}_{n=0}^{\infty}) \leq C_p(\{u_n\}_{n=0}^{\infty}).$$
This result was first obtained by Olevskii~\cite{Olevskii} and reads that the Haar basis in $L_p(0,1)$ has minimal unconditional constant among the others unconditional bases.
Almost two decades later the precise value of unconditional constant of the Haar basis was determined by Burkholder~\cite{Burkholder}:
$$C_p(\{h_n\}_{n=0}^{\infty}) = \max\left\{p, \frac{p}{p-1} \right\} -1,\:\: 1<p<\infty,$$
see also~\cite{Johnson2}.
Next we deduce on the basis of these remarkable results the following best estimate for constants for one dimensional unconditional Schauder decompositions in $L_p(0,1)$, $1<p<\infty$.
\begin{thm}\label{best}
  Let $\{\mathfrak{M}_n\}_{n=0}^{\infty}$ be an unconditional Schauder decomposition of $L_p(0,1)$, $1<p<\infty$, with constant $M$ (see Definition~\ref{MUSD}), such that for any $n$ we have $$\dim \mathfrak{M}_n=1.$$
  Then
  $$M\geq M_p=\frac12\max\left\{p, \frac{p}{p-1} \right\},$$
  where $M_p$ is the constant of unconditional Schauder decomposition corresponding to the Haar basis.
\end{thm}
\begin{proof}
  Since the Haar basis is the "best" in $L_p(0,1)$, $1<p<\infty$, in the mentioned above sense, we have for the a.s.c.p. $\{P_n\}_{n=0}^{\infty}$ of $\{\mathfrak{M}_n\}_{n=0}^{\infty}$, any
  choice of signs $\{\epsilon_i=\pm 1\}_{i=0}^{\infty}$ and any $x\in L_p$ that
  \begin{equation}\label{12}
    \left\| \sum\limits_{i=0}^{\infty} \epsilon_i P_i x\right\| \leq C\|x\|,
  \end{equation}
  where $C\geq C_p(\{h_n\}_{n=0}^{\infty}) $.
  Thus for any choice of selectors $\{\delta_i'\in\{0,1\}\}_{i=0}^{\infty}$
  there exist the set of signs $\{\epsilon_i'=\pm 1\}_{i=0}^{\infty}$  such that
  $$\delta_i'=\frac{1+\epsilon_i'}{2} $$ for all $i$.
   Hence, taking into account~(\ref{12}), for any $x\in L_p$ we have that
   \begin{align*}
      \left\| \sum\limits_{i=0}^{\infty} \delta_i' P_i x\right\| &= \left\| \frac12 \sum\limits_{i=0}^{\infty} \epsilon_i' P_i x+\frac12\sum\limits_{i=0}^{\infty}  P_i x\right\|\\
     & \leq  \frac12\left\| \sum\limits_{i=0}^{\infty} \epsilon_i' P_i x\right\|+\frac12 \|x\|\leq \frac{C+1}{2}\|x\|.
   \end{align*}
  Since $\{\mathfrak{M}_n\}_{n=0}^{\infty}$ is unconditional Schauder decomposition with constant $M$, we have on the other hand for any $x\in L_p$ that
  $$\left\| \sum\limits_{i=0}^{\infty} \delta_i' P_i x\right\|\leq M \|x\|.$$
  Subtracting from the latter inequality the previous inequality we arrive at
  $$M-\frac{C+1}{2} \geq 0.$$
  Thus
  $$M\geq \frac{C+1}{2} \geq \frac{C_p(\{h_n\}_{n=0}^{\infty})+1}{2}=\frac12\max\left\{p, \frac{p}{p-1} \right\}=M_p.$$
\end{proof}

Theorem~\ref{best} immediately yields that for any one dimensional unconditional Schauder decomposition of $L_p$ with constant $M$ we have
$$M\geq M_p=\frac12\max\left\{p, \frac{p}{p-1} \right\}>1$$
for any $1<p\neq 2$. Furthermore, it follows that the space $L_p$ has an unconditional basis with constant $1$ if and only if $p=2,$
despite the fact that the "best" Haar basis in $L_p(0,1)$, $1<p<\infty$,  is monotone and unconditional, hence Grinblyum-James orthogonal.
On the other hand, any slicing Schauder decomposition in $L_p$, $p\geq 1,$  each element of which is infinite dimensional,
is by Theorem~\ref{uncond} unconditional with constant $1.$


For Schauder-Orlicz decompositions in the space $C$ we have the following result, analogous to the Theorem~\ref{nonexistence}.
\begin{thm}\label{nonexistenceC}
  There is no Schauder-Orlicz decomposition in space $C$ with at least one finite dimensional subspace.
\end{thm}
\begin{proof}
Since any Schauder-Orlicz decomposition is unconditional by Theorem~\ref{uncond} and any
unconditional Schauder decomposition in $C$ is $c_0$-decomposition by~\cite{LindenstraussPelcz}, we have for
Schauder-Orlicz decomposition $\{\mathfrak{M}_n\}_{n=0}^{\infty}$ with the a.s.c.p.  $\{P_n\}_{n=0}^{\infty}$
and any $x\in C$
that
$$\max_n \|P_nx\| =\|x\|.$$
  The rest of the proof goes by lines of the proof of Theorem~\ref{nonexistence} and is based on the validity of the Daugavet equation in the space $C$, see~\cite{Daugavet}.
\end{proof}

Finally we have the following corollary of Theorem~\ref{nonexistenceC} concerning FDDs in space $C$.
\begin{corollary}\label{fddc}
  There is no Schauder-Orlicz decomposition in space $C$, which is an FDD.
\end{corollary}

\subsection{Nonexistence of Schauder-Orlicz decompositions with at least one finite dimensional subspace in Banach spaces with pseudo-Daugavet property}

A Banach space $X$ has a pseudo-Daugavet property~\cite{Popov},~\cite{PopovB} if there exists a strictly increasing function $\varphi_{X}:$ $(0,\infty)\mapsto (0,\infty)$ such that for any nonzero
compact operator $K$ on $X$ we have
$$\|I+K\|\geq 1+\varphi_{X}\left(\|K\|\right).$$
In particular, if $X$ has a pseudo-Daugavet property with $\varphi_{X}(t)=t$, then it has the Daugavet property.
Historically first appeared examples of spaces with Daugavet property are classical spaces $C$ and $L_1$, see~\cite{Daugavet} and~\cite{Babenko}.
It is worth mentioning that Kadets first proved that spaces with Daugavet property have no unconditional basis~\cite{Kadets1}.
For more about the Daugavet property, its relation to the geometry of a Banach space and its variations we refer, e.g. to monographs~\cite{Kadets} and~\cite{PopovB}.
Note that a pseudo-Daugavet property is related to the problem of best compact approximation in $X$~\cite{Axler},~\cite{Benyamini}.
In~\cite{Benyamini} Benyamini and Lin proved that spaces $L_p$, $1< p\neq 2$ enjoy a pseudo-Daugavet property.
The validity of pseudo-Daugavet property was also proved for Lorentz function spaces $L_{w, p} \left([0,1], \mu \right)$, $p>2$, without any restrictions on function $w$~\cite{Popov},~\cite{PopovB}, for non-atomic non-commutative $L_p$-spaces, where $1< p\neq 2$~\cite{Oikberg}, and for certain spaces of operators on $\ell_p$~\cite{Oikberg2}.

We generalize our Theorems~\ref{nonexistence} and~\ref{nonexistenceC} to the case of Banach spaces with pseudo-Daugavet property as follows.
\begin{thm}\label{nonexistenceX}
The following statements are true.

(i) Let $X$ be a Banach space satisfying pseudo-Daugavet property.
Then there is no Schauder-Orlicz decomposition in space $X$ with at least one finite dimensional subspace.

(ii) Conversely, let $X$ be a Banach space with Schauder-Orlicz decomposition, at least one subspace of which is finite dimensional.
Then $X$ does not have pseudo-Daugavet property.
\end{thm}
\begin{proof}
To prove (i) suppose that there exists the Schauder-Orlicz decomposition  $\{\mathfrak{M}_n\}_{n=0}^{\infty}$ with an Orlicz function $\Phi$ in $X$, such that
  \begin{equation}\label{dimX}
    \dim \mathfrak{M}_0<\infty.
  \end{equation}
Denote by   $\{P_n\}_{n=0}^{\infty}$ the a.s.c.p. of $\{\mathfrak{M}_n\}_{n=0}^{\infty}$.
Note fisrt that~(\ref{dim}) means that $P_0$ is compact.
Then, on the one hand, since $X$ has pseudo-Daugavet property we have that there exists a strictly increasing function $\varphi_{X}:$ $(0,\infty)\mapsto (0,\infty)$ such that
\begin{equation}\label{esX}
  \left\|I-P_0\right\|= \left\|\sum\limits_{n=0}^{\infty} P_n -P_0\right\|= \left\|\sum\limits_{n=1}^{\infty} P_n \right\|\geq 1+ \varphi_{X}\left(\|P_0\|\right)>1.
\end{equation}

On the other hand, since $\{\mathfrak{M}_n\}_{n=0}^{\infty}$ is a Schauder-Orlicz decomposition with an Orlicz function $\Phi$ and the a.s.c.p. $\{P_n\}_{n=0}^{\infty}$, for any $x\in X$ we have
$$\inf\left\{\rho >0: \sum\limits_{n=0}^{\infty} \Phi\left(\frac{\|P_n x\|}{\rho}\right)\leq 1\right\} = \|x\|.$$
Hence for any $x\in X$ we have
\begin{align*}
  \left\|\sum\limits_{n=1}^{\infty} P_n x \right\| & = \inf\left\{\rho >0: \sum\limits_{j=0}^{\infty} \Phi\left(\frac{\left\|P_j \left(\sum\limits_{n=1}^{\infty} P_n x \right)\right\|}{\rho}\right)\leq 1\right\}\\
   & = \inf\left\{\rho >0: \sum\limits_{n=1}^{\infty} \Phi\left(\frac{\|P_n x\|}{\rho}\right)\leq 1\right\}\\
   & \leq \inf\left\{\rho >0: \sum\limits_{n=0}^{\infty} \Phi\left(\frac{\|P_n x\|}{\rho}\right)\leq 1\right\} =\|x\|.
\end{align*}
It means that
\begin{equation}\label{esX2}
  \left\|\sum\limits_{n=1}^{\infty} P_n \right\| \leq 1.
\end{equation}
Combining~(\ref{esX}) with (\ref{esX2}) we observe a contradiction.

To prove (ii) we consider Schauder-Orlicz decomposition  $\{\mathfrak{M}_n\}_{n=0}^{\infty}$ with an Orlicz function $\Phi$ in $X$, such that condition~(\ref{dimX}) holds.
By   $\{P_n\}_{n=0}^{\infty}$ we denote the a.s.c.p. of $\{\mathfrak{M}_n\}_{n=0}^{\infty}$.
Assume that $X$ has pseudo-Daugavet property. Hence
there exists a strictly increasing function $\varphi_{X}:$ $(0,\infty)\mapsto (0,\infty)$ such that
\begin{equation}\label{e1}
  \left\|I-P_0\right\|\geq 1+ \varphi_{X}\left(\|P_0\|\right)>1.
\end{equation}

On the other hand, since $\{\mathfrak{M}_n\}_{n=0}^{\infty}$ is a Schauder-Orlicz decomposition with an Orlicz function $\Phi$ and the a.s.c.p. $\{P_n\}_{n=0}^{\infty}$,
we have, as in the proof of (i), that condition~(\ref{esX2}) holds.
Thus we arrive at a contradiction between~(\ref{e1}) and~(\ref{esX2}).
\end{proof}

Combining Theorem~\ref{nonexistenceX} with~\cite{Popov},~\cite{Oikberg} and~\cite{Oikberg2}   we have the following immediate corollary.
\begin{corollary}\label{pDpcor}
Lorentz function spaces $L_{w, p} \left([0,1], \mu \right)$, $p>2$, non-atomic non-commutative $L_p$-spaces, where $1< p\neq 2$ and certain spaces of operators on $\ell_p$ from~\cite{Oikberg2}
have no Schauder-Orlicz decomposition, at least one subspace of which is finite dimensional.
\end{corollary}

We have the following corollary, extending Corollaries~\ref{fddL} and~\ref{fddc}.
\begin{corollary}\label{fddX}
Let $X$ be a Banach space with pseudo-Daugavet property.
  Then there is no Schauder-Orlicz decomposition in $X$, which is an FDD.
  Conversely, if there exists a Schauder-Orlicz decomposition in $X$, which is an FDD, then $X$ has no pseudo-Daugavet property.
\end{corollary}

Finally, in the class of spaces with Schauder-Orlicz decompositions combination of Theorem~\ref{nonexistenceX} with~Corollary~\ref{fddX} turns into the following criterion.
\begin{thm}\label{crit}
  Let $X$ be a Banach space, which has a Schauder-Orlicz decomposition. Then the following assertions are equivalent.
  \begin{enumerate}
    \item A Banach space $X$ has pseudo-Daugavet property.
    \item There is no Schauder-Orlicz decomposition in $X$ with at least one finite dimensional subspace.
    \item There is no Schauder-Orlicz decomposition in $X$, which is an FDD.
  \end{enumerate}
\end{thm}

\section{Concluding remarks and perspectives}

Since the concept of Schauder-Orlicz decomposition in Banach spaces seems to be new there are many interesting questions for further study.
Among them is also a question what the existence of Schauder-Orlicz decomposition in a  Banach space $X$ mean to geometric properties of $X$?
Theorem~\ref{crit} gives only partial answer and characterizes pseudo-Daugavet property, which has of course geometric meaning, for Banach spaces possessing Schauder-Orlicz decompositions on a language of sizes of dimensions of these decompositions.

We note some of questions for further study, which lie on a surface.

1) Does non-separable Orlicz spaces $\ell_{\Phi}$, i.e. such that $\Phi$ does not satisfy the $\Delta_2$-condition at zero~\cite{L-T}, have Schauder decomposition?
By~\cite{Chadwick} we have that all reflexive non-separable $\ell_{\Phi}$ have a Schauder decomposition, but what with non-reflexive non-separable $\ell_{\Phi}$?
Moreover, does non-separable Orlicz spaces $\ell_{\Phi}$ have $\ell_{\Phi}$-decomposition and Schauder-Orlicz decomposition?

2) It is proved in~\cite{LindenstraussPelcz} that all unconditional Schauder decompositions in $\ell_1$ are $\ell_1$-decompositions
and all unconditional Schauder decompositions in $c_0$ are
$c_0$-decompositions. Thus the class of unconditional Schauder decompositions in $\ell_1$ (resp. $c_0$) coincides with the class of all $\ell_1$-decompositions (resp. $c_0$-decompositions).
However, an interesting question on the uniqueness, up to an isomorphism, of $\ell_1$-decomposition with preassigned dimensions in $\ell_1$ and
 of $c_0$-decomposition with preassigned dimensions in $c_0$ is still open.
On the way of finding the answer we propose the following.
\begin{conjecture}\label{conj}
Every Schauder-Orlicz decomposition with preassigned dimensions in $\ell_1$ is unique, up to an isometry.
Every Schauder-Orlicz decomposition with preassigned dimensions in $c_0$ is unique, up to an isometry.
\end{conjecture}

3) Note that in~\cite{Singer2} Schauder decomposition satisfying the condition $2$ of Theorem~\ref{orthogonal} also is called orthogonal.
This notion of orthogonality is due to Singer, dates back to 1962~\cite{Singer3} and has a geometric meaning in view of the notion of Birkhoff-orthogonality~\cite{Birkhoff},~\cite{Singer3}, see also~\cite{Singer1}.
However, in view of the proof of Proposition~\ref{schauder basis}, notions of Grinblyum-James orthogonality and Singer orthogonality are distinct, since every Singer-orthogonal basis is unconditional~\cite{Singer3} and unconditional bases do not exist in $C[a,b]$.
Nevertheless a basis is Singer-orthogonal if and only if every its permutation is Grinblyum-James orthogonal~\cite{Singer1}. So Singer-orthogonality is much stronger, then
Grinblyum-James orthogonality.
On the other hand by Theorem~\ref{orthogonal} each Schauder-Orlicz decomposition is Singer-orthogonal.
However, the whole picture of mutual relations of this two concepts of orthogonality and the concept of Schauder-Orlicz decomposition looks not complete.

4) Clearly if $\{\mathfrak{M}_n\}_{n=0}^{\infty}$ is a Schauder-Orlicz decomposition in $X$ with an Orlicz function $\Phi$, then $\{S \mathfrak{M}_n\}_{n=0}^{\infty}$, where $S:X\mapsto X$ is an isomorphism, is an $\ell_{\Phi}$-decomposition in $X$.
However, does every space with $\ell_{\Phi}$-decomposition has a Schauder-Orlicz decomposition?
The space $C[0,1]$ looks as the key pointer here, since it has a $c_0$-decomposition~\cite{Singer2}, but the construction of Schauder-Orlicz decomposition in  $C[0,1]$ does not look easy.
Taking into account Theorems~\ref{uncond} and~\ref{nonexistenceC}, the first step may be the construction of unconditional Schauder decomposition in  $C[0,1]$ with constant $1,$
each subspace of which is infinite dimensional.

5) One can prove, applying Theorem 15.3 from~\cite{Burkholder} (see also~\cite{Johnson2}) and arguments from the proof of Theorem~\ref{best},
that every monotone Schauder decomposition in $L_p$, $1< p <\infty$, is unconditional with constant
$$M_p=\frac12\max\left\{p, \frac{p}{p-1} \right\}.$$
Combining this fact with Theorem~\ref{best} we have that any Grinblyum-James orthogonal basis in $L_p$, $1< p <\infty$, is unconditional with constant $M_p$
and this constant is the best possible.
Therefore it is natural to ask about the best possible unconditional constant of any unconditional FDD in $L_p$, $1<p\neq 2$, in an intermediate situation between Grinblyum-James orthogonal unconditional bases in $L_p$ with best possible constant $M_p$ and slicing Schauder-Orlicz decompositions with unconditional constant $1$.

6) Does the Lorentz space $L_{w, p} \left([0,1], \mu \right)$, where $1\leq p\neq 2$  and $w(t)\neq 1$, have a Schauder-Orlicz decomposition?
By Theorem~\ref{nonexistenceX} it may have only Schauder-Orlicz decompositions, each subspace of which is infinite dimensional.
The same question concerns non-atomic non-commutative $L_p$-spaces, where $1 \leq p\neq 2$, and spaces of operators on $\ell_p$, considered in~\cite{Oikberg2}.

\section*{Acknowledgments}
The author is grateful to Prof. Vladimir Kadets for the discussion during the
International Conference in Functional Analysis dedicated to $125^{th}$ anniversary
of S.~Banach in Lviv and pointing out paper~\cite{Benyamini}.

\section*{Data availability}
This manuscript has no associated data.

\section*{Conflict of interest}
The author state that there is no conflict of interest.




\bibliographystyle{elsarticle-num}
\bibliography{SD}





\end{document}